\documentclass[11pt, reqno]{amsart}
\usepackage{indentfirst, amssymb, amsmath, amsthm, mathrsfs, setspace, indentfirst, enumerate,  mathrsfs, amsmath, amsthm}
\usepackage[bookmarksnumbered, colorlinks, plainpages]{hyperref}
\usepackage{mathrsfs}
\usepackage{cite}
\newtheorem*{theoA}{Theorem A}
\newtheorem*{theoB}{Theorem B}
\newtheorem*{theoC}{Theorem C}
\newtheorem*{theoD}{Theorem D}
\newtheorem*{theoE}{Theorem E}

\newtheorem*{conjA}{Conjecture A}
\newtheorem*{conjB}{Conjecture B}

\newtheorem{theo}{Theorem}[section]
\newtheorem{lem}{Lemma}[section]

\newtheorem{exm}{Example}[section]

\newtheorem*{quesA}{Question A}
\newtheorem{conj}{Conjecture}[section]
\newcommand{\ol}{\overline}

\newcommand{\be}{\begin{equation}}
\newcommand{\ee}{\end{equation}}
\newcommand{\beas}{\begin{eqnarray*}}
\newcommand{\eeas}{\end{eqnarray*}}
\newcommand{\bea}{\begin{eqnarray}}
\newcommand{\eea}{\end{eqnarray}}

\numberwithin{equation}{section}
\begin{document}

\title[B\MakeLowercase{r\"{u}ck conjecture for solutions of first-order partial differential equations}......]{\LARGE B\Large\MakeLowercase{r\"{u}ck conjecture for solutions of first-order partial differential equations in} $\mathbb{C}^m$}

\date{}
\author[S. M\MakeLowercase{ajumder}, N. S\MakeLowercase{arkar} \MakeLowercase{and} D. P\MakeLowercase{ramanik}]{S\MakeLowercase{ujoy} M\MakeLowercase{ajumder}$^*$, N\MakeLowercase{abadwip} S\MakeLowercase{arkar} \MakeLowercase{and} D\MakeLowercase{ebabrata} P\MakeLowercase{ramanik}}
\address{Department of Mathematics, Raiganj University, Raiganj, West Bengal-733134, India.}
\email{sm05math@gmail.com}
\address{Department of Mathematics, Raiganj University, Raiganj, West Bengal-733134, India.}
\email{naba.iitbmath@gmail.com}
\address{Department of Mathematics, Raiganj University, Raiganj, West Bengal-733134, India.}
\email{debumath07@gmail.com}

\renewcommand{\thefootnote}{}
\footnote{\emph{Corresponding Author}: Sujoy Majumder.}

\renewcommand{\thefootnote}{\arabic{footnote}}
\setcounter{footnote}{0}

\begin{abstract} In this paper, we study the Br\"{u}ck conjecture \cite{Bruck-1996} by interpreting it through solutions of first-order partial differential equations in several complex variables. Our results show that the Br\"{u}ck conjecture \cite{Bruck-1996} 
in $\mathbb{C}^m$ holds under certain additional conditions. In pursuit of this objective, we also establish a Borel-Caratheodory theorem in $\mathbb{C}^m$ and derive several fundamental results on the order and hyper-order of entire functions in higher dimensions.
\end{abstract}

\thanks{Typeset by \AmS -\LaTeX}
\maketitle

\section{{\bf Introduction}}
We define $\mathbb{Z}_+=\mathbb{Z}[0,+\infty)=\{n\in \mathbb{Z}: 0\leq n<+\infty\}$ and $\mathbb{Z}^+=\mathbb{Z}(0,+\infty)=\{n\in \mathbb{Z}: 0<n<+\infty\}$.
On $\mathbb{C}^m$, we define
\begin{align*}
\partial_{z_i}=\frac{\partial}{\partial z_i},\ldots, \partial_{z_i}^{l_i}=\frac{\partial^{l_i}}{\partial z_i^{l_i}}\;\;\text{and}\;\;\partial^{I}=\frac{\partial^{|I|}}{\partial z_1^{i_1}\cdots \partial z_m^{i_m}}
\end{align*}
where $l_i\in \mathbb{Z}^+\;(i=1,2,\ldots,m)$ and $I=(i_1,\ldots,i_m)\in\mathbb{Z}^m_+$ be a multi-index such that $|I|=\sum_{j=1}^m i_j$.
\smallskip

We firstly recall some basic notions in several complex variables (see \cite{HLY,Noguchi-Winkelmann,WS}).
On $\mathbb{C}^m$, the exterior derivative $d$ splits $d= \partial+ \bar{\partial}$ and twists to $d^c= \frac{\iota}{4\pi}\left(\bar{\partial}- \partial\right)$. Clearly $dd^{c}= \frac{\iota}{2\pi}\partial\bar{\partial}$. A non-negative function $\tau: \mathbb{C}^m\to \mathbb{R}[0,b)\;(0<b\leq \infty)$ of class $\mathbb{C}^{\infty}$ is said to be an exhaustion of $\mathbb{C}^m$ if $\tau^{-1}(K)$ is compact whenever $K$ is. 
An exhaustion $\tau_m$ of $\mathbb{C}^m$ is defined by $\tau_m(z)=||z||^2$. The standard Kaehler metric on $\mathbb{C}^m$ is given by $\upsilon_m=dd^c\tau_m>0$. On $\mathbb{C}^m\backslash \{0\}$, we define $\omega_m=dd^c\log \tau_m\geq 0$ and $\sigma_m=d^c\log \tau_m \wedge \omega_m^{m-1}$. For any $S\subseteq \mathbb{C}^m$, let $S[r]$, $S(r)$ and $S\langle r\rangle$ be the intersection of $S$ with respectively the closed ball, the open ball, the sphere of radius $r>0$ centered at $0\in \mathbb{C}^m$.

\smallskip
Let $f$ be a holomorphic function on $G(\not=\varnothing)$, where $G$ is an open subset of $\mathbb{C}^m$. Then we can write $f(z)=\sum_{i=0}^{\infty}P_i(z-a)$, where the term $P_i(z-a)$ is either identically zero or a homogeneous polynomial of degree $i$. Certainly the zero multiplicity $\mu^0_f(a)$ of $f$ at a point $a\in G$ is defined by $\mu^0_f(a)=\min\{i:P_i(z-a)\not\equiv 0\}$.

\medskip
Let $f$ be a meromorphic function on $G$. Then there exist holomorphic functions $g$ and $h$ such that $hf=g$ on $G$ and $\dim_z h^{-1}(\{0\})\cap g^{-1}(\{0\})\leq m-2$. Therefore the $c$-multiplicity of $f$ is just $\mu^c_f=\mu^0_{g-ch}$ if $c\in\mathbb{C}$ and $\mu^c_f=\mu^0_h$ if $c=\infty$. The function $\mu^c_f: \mathbb{C}^m\to \mathbb{Z}$ is nonnegative and is called the $c$-divisor of $f$. If $f\not\equiv 0$ on each component of $G$, then $\nu=\mu_f=\mu^0_f-\mu^{\infty}_f$ is called the divisor of $f$. We define 
$\text{supp}\; \nu=\text{supp}\;\mu_f=\ol{\{z\in G: \nu(z)\neq 0\}}$.

For $t>0$, the counting function $n_{\nu}$ is defined by
\begin{align*}
 n_{\nu}(t)=t^{-2(m-1)}\int_{A[t]}\nu \upsilon_m^{m-1},
 \end{align*}
where $A=\text{supp}\;\nu$. The valence function of $\nu$ is defined by 
\begin{align*}
N_{\nu}(r)=N_{\nu}(r,r_0)=\int_{r_0}^r n_{\nu}(t)\frac{dt}{t}\;\;(r\geq r_0).
\end{align*}

For $a\in\mathbb{P}^1$, we write $n_{\mu_f^a}(t)=n(t,a;f)$, if $a\in\mathbb{C}$ and $n_{\mu_f^a}(t)=n(t,f)$, if $a=\infty$. Also we write $N_{\mu_f^a}(r)=N(r,a;f)$ if $a\in\mathbb{C}$ and $N_{\mu_f^a}(r)=N(r,f)$ if $a=\infty$.
For $k\in\mathbb{N}$, define the truncated multiplicity functions on $\mathbb{C}^m$ by $\mu_{f,k}^a(z)=\min\{\mu_f^a(z),k\}$,
and write the truncated counting functions $n_{\nu}(t)=n_k(t,a;f)$, if $\nu=\mu_{f,k}^a$ and $n_{\nu}(t)=\ol{n}(t,a;f)$, if $\nu=\mu_{f,1}^a$. Also we write
 $N_{\nu}(t)=N_k(t,a;f)$, if $\nu=\mu_{f,k}^a$ and $N_{\nu}(t)=\ol{N}(t,a;f)$, if $\nu=\mu_{f,1}^a$.

\medskip
With the help of the positive logarithm function, we define the proximity function of $f$ by
\begin{align*}
m(r, f)=\mathbb{C}^m\langle r; \log^+ | f | \rangle=\int_{\mathbb{C}^m\langle r\rangle} \log^+ |f|\;\sigma_m.
\end{align*}

The characteristic function of $f$ is defined by $T(r, f)=m(r,f)+N(r,f)$. We define $m(r,a;f)=m(r,f)$ if $a=\infty$ and $m(r,a;f)=m(r,1/(f-a))$ if $a$ is finite complex number. Now if $a\in\mathbb{C}$, then the first main theorem of Nevanlinna theory states that $m(r,a;f)+N(r,a;f)=T(r,f)+O(1)$, where $O(1)$ denotes a bounded function when $r$ is sufficiently large. We define the order and the hyper-order of $f$ by
\begin{align*}
\rho(f):=\limsup_{r \rightarrow \infty} \frac{\log T(r, f)}{\log r}\;\text{and}\;\rho_1(f):=\limsup_{r \rightarrow \infty} \frac{\log \log T(r, f)}{\log r}.
\end{align*}

Let $S(f)=\{g:\mathbb{C}^m\to\mathbb{P}^1\;\text{meromorphic}:\parallel T(r,g)=o(T(r,f))\}$, where $\parallel$ indicates that the equality holds only outside a set of finite measure on $\mathbb{R}^+$ and the element in $S(f)$ is called the small function of $f$.

Let $f$, $g$ and $a$ be meromorphic functions on $\mathbb{C}^m$. Then one can find three pairs
of entire functions $f_1$ and $f_2$, $g_1$ and $g_2$, and $a_1$ and $a_2$, in which each pair is coprime
at each point in $\mathbb{C}^m$ such that $f = f_2/f_1$, $g=g_2/g_1$ and $a = a_2/a_1$.
We say that $f$ and $g$ share $a$ by counting multiplicities (CM) if $\mu_{a_1f_2-a_2f_1}^0=\mu_{a_1g_2-a_2g_1}^0\;(a\not\equiv \infty)$ and $\mu_{f_1}^0=\mu_{g_1}^0\;\;(a=\infty)$. 

\medskip
Rubel and Yang \cite{Rubel-Yang} first considered the uniqueness of an entire function in $\mathbb{C}$ when it shares two values CM with its first derivative. In 1977 they proved:

\begin{theoA}\cite{Rubel-Yang} Let $f$ be a non-constant entire function in $\mathbb{C}$ and let $a$ and $b$ be two distinct finite complex numbers. If $f$ and $f^{(1)}$ share $a$ and $b$ CM, then $f\equiv f^{(1)}$.
\end{theoA} 

In the following result, Mues and Steinmetz \cite{Mues-Steinmetz} generalized \emph{Theorem A} from sharing values CM to IM.

\begin{theoB}\cite{Mues-Steinmetz} Let $f$ be a non-constant entire function in $\mathbb{C}$ and let $a$ and $b$ be two distinct finite complex numbers. If $f$ and $f^{(1)}$ share $a$ and $b$ IM, then $f\equiv f^{(1)}$.
\end{theoB}

\smallskip
In recent years, the Nevanlinna value distribution theory in several complex variables has emerged as a prominent and rapidly growing area of research in complex analysis. This field has garnered significant attention due to its deep theoretical insights and wide-ranging applications in mathematics and related disciplines. Researchers have been particularly intrigued by its potential to extend classical results from one complex variable to higher-dimensional settings, as a result, this topic has become a focal point for contemporary studies in several complex variables. These works highlight both theoretical developments and applications in complex geometry, normal families, linear partial differential equations, partial difference equations, partial differential-difference equations, and Fermat-type functional equations. These references \cite{BM1}, \cite{Berenstein-Chang-Li}-\cite{HZ1}, \cite{Hu-Yang-2014}, \cite{IM}-\cite{LZ}, \cite{MD1}, \cite{MDP}, \cite{MS}, \cite{MS1}, \cite{MSP}, \cite{GS2}, \cite{XC1}- \cite{XW1} provide a foundation for understanding the current state of research in Nevanlinna value distribution theory in several complex variables.

Let $f$ be a non-constant entire function in $\mathbb{C}^m$ and 
\begin{align}\label{do2} 
L=D^{(n)}+D^{(n-1)}+\ldots+D^{(1)}+D^{(0)}
\end{align}
be a partial differential operator, where
$D^{(j)}=\sideset{}{_{|I|=j}}{\sum} a_I\partial^{I}$ and $a_I\in S(f)$.

In 1996, Berenstein et. al. \cite{Berenstein-Chang-Li} proved that a non-constant entire function $f$ in $\mathbb{C}^m$ must be a solution of the partial differential equation of $L(w)-w=0$, i.e., $f$ must be identically equal to its partial differential polynomial $L(f)$ if $f$ and $L(f)$ share $a_1$ and $a_2$ CM, where $a_1, a_2\in S(f)$ such that $a_1\not\equiv a_2$. They proved the following result.

\begin{theoC} \cite[Theorem 2.2]{Berenstein-Chang-Li} Let $f$ be a non-constant entire function in $\mathbb{C}^m$ and let $n$ be a positive integer such that $L(f)\not\equiv 0$, where $L$ is defined by (\ref{do2}). If $f$ and $L(f)$ share $a_1$ and $a_2$ CM, where $a_1, a_2\in S(f)$ such that $a_1\not\equiv a_2$, then $f\equiv L(f)$. 
\end{theoC}

Now in the context of sharing one value, the following question creates a new era.
\begin{quesA} What conclusion can be made if $f$ be a non-constant entire function on $\mathbb{C}$ shares only one value with $f^{(1)}$?\end{quesA}

Inspired by \emph{Question A}, in 1996, Br\"{u}ck \cite{Bruck-1996} proposed the following conjecture.
\begin{conjA}\cite{Bruck-1996} Let $f$ be a non-constant entire function in $\mathbb{C}$ such that $\rho_1(f)\not\in\mathbb{N}\cup\{\infty\}$ and $a\in\mathbb{C}$. If $f$ and $f^{(1)}$ share $a$ CM, then 
\begin{align}\label{aa1}
f^{(1)}-a=c(f-a),
\end{align}
where $c$ is a non-zero constant.
\end{conjA}

It is easy to verify that all the solutions of (\ref{aa1}) takes the form
\begin{align}\label{ab1}
f(z)=c_1e^{cz}+a-\frac{a}{c},
\end{align}
where $c_1$ is a non-zero constant. Since $f$ and $f^{(1)}$ share $a$ CM in \emph{Conjecture A}, there exists an entire function $\alpha$ in $\mathbb{C}$ such that
\begin{align}\label{aa2}
\frac{f^{(1)}(z)-a}{f(z)-a}=e^{\alpha(z)}.
\end{align}

Therefore in order to resolve \emph{Conjecture A}, we have to prove that $\alpha$ reduces to a constant. As a result if $\alpha$ is a transcendental entire function or a non-constant polynomial in (\ref{aa2}), then \emph{Conjecture A} does not hold. On the other hand we see that \emph{Conjecture A} may not be true if we assume that $\rho(f)=+\infty$ as all the solutions of (\ref{aa1}) are given by (\ref{ab1}), where we see that $\rho(f)=1$. Therefore \emph{Conjecture A} can be re-stated as follows:

\begin{conjB} Let $f$ be a non-constant entire function in $\mathbb{C}$ such that $\rho_1(f)\not\in\mathbb{N}\cup\{\infty\}$ and $a\in\mathbb{C}$. If $f^{(1)}-a=e^{\alpha}(f-a)$, where $\alpha$ is an entire function in $\mathbb{C}$, then $\alpha$ reduces to a constant, $d$ say and $f(z)$ takes the form $f(z)=c_1e^{cz}+a-\frac{a}{c}$, where $c=e^d$ and $c_1$ are non-zero constant.
\end{conjB}

Br\"{u}ck \cite{Bruck-1996} himself demonstrated that \emph{Conjecture A} does not hold when $\rho_1(f)\in\mathbb{N}\cup\{\infty\}$, by constructing solutions of the following differential equations:

\begin{align*}
 \frac{f^{(1)}(z)-a}{f(z)-a}=e^{z^{n}},
 \end{align*}
where $\rho_1(f)=n\in\mathbb{N}$ and   
\begin{align*}
\frac{f^{(1)}(z)-a}{f(z)-a}=e^{e^{z}},
\end{align*}
where $\rho_1(f)=\infty$.

\emph{Conjecture A} for the special case $a=0$ had been resolved by Br\"{u}ck \cite{Bruck-1996} as follows.

\begin{theoD}\cite{Bruck-1996} Let $f$ be a non-constant entire function on $\mathbb{C}$ such that $\rho_1(f)\not\in\mathbb{N}\cup\{\infty\}$. If $f$ and $f^{(1)}$ share $0$ CM, then $f^{(1)}=cf$, where $c$ is a non-zero constant and $f(z)$ takes the form $f(z)=c_1e^{cz}$, where $c_1$ is a non-zero constant.
\end{theoD}

In the same paper, Br\"{u}ck \cite{Bruck-1996} proved the following result, which demonstrates that the growth condition on $f$ in 
\emph{Conjecture A} can be removed provided that $N(r,0;f^{(1)})=o(T(r,f))$.
\begin{theoE}\cite{Bruck-1996} Let $f$ be a non-constant entire function on $\mathbb{C}$ such that $N(r,0;f^{(1)})=o(T(r,f))$. If $f$ and $f^{(1)}$ share $a$ CM, then $f^{(1)}-1=c(f-1)$, where $c$ is a non-zero constant and $f(z)$ takes the form
$f(z)=c_1e^{cz}+a-\frac{a}{c}$, where $c_1$ is a non-zero constant. 
\end{theoE}

Now motivated by \emph{Conjecture B}, we suggest to extend \emph{Conjecture B} into several complex variables as follows:

\begin{conj} Let $f$ be a non-constant entire function in $\mathbb{C}^m$ such that $\rho_1(f)\not\in\mathbb{N}\cup\{\infty\}$ and $a\in\mathbb{C}$. If 
\begin{align}\label{bb1}
\partial_{z_i}(f(z))-a=e^{\alpha(z)}(f(z)-a)
\end{align}
 for all $i\in\mathbb{Z}[1,m]$, where $\alpha(z)$ is an entire function in $\mathbb{C}^m$ and $a$ is a finite complex number, then $\alpha(z)$ reduces to a constant, $c$ say and
\begin{align*}
f(z)=c_1e^{A(z_1+ \cdots+z_m)}+a-\frac{a}{A},
\end{align*}
where $A=e^c$ and $c_1$ are non-zero constant.
\end{conj}

In the following two examples, we can verify that {\emph Conjecture 1.1} does not hold when $\rho_1(f)\in\mathbb{N}\cup\{\infty\}$.

\begin{exm} Let 
\begin{align*}
f(z_1,\ldots,z_m)=e^{e^{z_1+\cdots+z_m}}\int_{0}^{z_1+\cdots+z_m}e^{-e^t}(1-e^t)dt.
\end{align*}

Clearly $\rho_1(f)=1$. Note that for all $i\in\mathbb{Z}[1,m]$, we have
\begin{align*}
\partial_{z_i}(f(z))=e^{z_1+\cdots+z_m}(f(z)-1)+1
\end{align*}
and so for all $i\in\mathbb{Z}[1,m]$, we get
\begin{align*}
\partial_{z_i}(f(z))-1=e^{z_1+\cdots+z_m}(f(z)-1).
\end{align*}
\end{exm}

\begin{exm} Let
\begin{align*}
f(z_1,\ldots,z_m)=e^{\beta(z)}\int_{0}^{z_1+\cdots+z_m}e^{-\beta(z)}(1-e^{e^t})dt,
\end{align*}
where $\beta(z)=\int_{0}^{z_1+\cdots+z_m}e^{e^t}dt$. Clearly $\rho_1(f)=+\infty$. Note that for all $i\in\mathbb{Z}[1,m]$, we have
\begin{align*}
\partial_{z_i}(f(z))=e^{e^{z_1+\cdots+z_m}}(f(z)-1)+1
\end{align*}
and so for all $i\in\mathbb{Z}[1,m]$, we get
\begin{align*}
\partial_{z_i}(f(z))-1=e^{e^{z_1+\cdots+z_m}}(f(z)-1).
\end{align*}
\end{exm}

Following example shows that \emph{Conjecture 1.1} does not holds if $e^{\alpha(z)}$ is replaced by an entire function having zeros in 
(\ref{bb1}). 
\begin{exm} Let 
\begin{align*}
f(z_1,\ldots,z_m)=e^{\frac{(z_1+\cdots+z_m)^2}{2}}\left(\int_{0}^{z_1+\cdots+z_m}e^{-\frac{t^2}{2}}(1-t)dt+1\right).
\end{align*}

Note that for all $i\in\mathbb{Z}[1,m]$, we have
\begin{align*}
\partial_{z_i}(f(z))=(z_1+\cdots+z_m)f(z)+1-(z_1+\cdots+z_m)
\end{align*}
and so for all $i\in\mathbb{Z}[1,m]$, we get
\begin{align*}
\partial_{z_i}(f(z))-1=(z_1+\cdots+z_m)(f(z)-1).
\end{align*}
\end{exm}

Our first result shows that \emph{Conjecture 1.1} holds when $a=0$.
\begin{theo}\label{t3.1} Let $f$ be a non-constant entire function in $\mathbb{C}^m$ such that $\rho_1(f)\not\in\mathbb{N}\cup\{\infty\}$. If 
\begin{align*}
\partial_{z_i}(f(z))=e^{\alpha(z)}f(z)
\end{align*}
for all $i\in\mathbb{Z}[1,m]$, where $\alpha(z)$ is an entire function in $\mathbb{C}^m$, then $\alpha(z)$ reduces to a constant, $c$ say and 
\begin{align*}
f(z_1,\ldots,z_m)=c_1e^{A(z_1+ \cdots+z_m)},
\end{align*}
where $A=e^c$ and $c_1$ are non-zero constant.
\end{theo}

Our second result shows that \emph{Conjecture 1.1} holds under the following additional condition 
\begin{align*}
\parallel N(r,0;\partial_{z_i}(f))=o(T(r,f))
\end{align*}
for all $i\in\mathbb{Z}[1,m]$. However, in our second result we can drop the hypothesis on the growth of $f$.

\begin{theo}\label{t3.2} Let $f(z)$ be a non-constant entire function in $\mathbb{C}^m$ such that $\parallel N(r,0;\partial_{z_i}(f))=o(T(r,f))$ for all $i\in\mathbb{Z}[1,m]$. If 
\begin{align*}
\partial_{z_i}(f(z))-a=e^{\alpha(z)}(f(z)-a),
\end{align*}
for all $i\in\mathbb{Z}[1,m]$, where $\alpha(z)$ is an entire function in $\mathbb{C}^m$ and $a$ is a non-zero constant, then $\alpha(z)$ reduces to a constant, $c$ say and
\begin{align*}
f(z)=c_1e^{A(z_1+ \cdots+z_m)}+a-\frac{a}{A},
\end{align*}
where $A=e^c$ and $c_1$ are non-zero constant.
\end{theo}

\section{\bf{Auxiliary Lemmas}}
\begin{lem}\label{L.1}\cite[Lemma 1.37]{HLY} Let $f$ be a non-constant meromorphic function in $\mathbb{C}^m$ and let  $I=(\alpha_1,\alpha_2,\ldots,\alpha_m)\in \mathbb{Z}^m_+$ be a multi-index. Then for any $\varepsilon>0$, we have
\begin{align*}
\parallel\;m\left(r,\frac{\partial^I(f)}{f}\right)\leq |I|\log^+T(r,f)+|I|(1+\varepsilon)\log^+\log T(r,f)+O(1).
\end{align*}
\end{lem}

\begin{lem}\label{L.2}\cite[Lemma 1.2]{hy2} Let $f$ be a non-constant meromorphic function in $\mathbb{C}^m$ and let $a_1,a_2,\ldots,a_q$ be different points in $\mathbb{C}\cup\{\infty\}$. Then
\begin{align*}
 \parallel (q-2)T(r,f)\leq \sideset{}{_{j=1}^{q}}{\sum} \ol N(r,a_j;f)+O(\log (rT(r,f))).
 \end{align*}
\end{lem}

\begin{lem}\label{L.6} Let $f$ be a non-constant meromorphic function in $\mathbb{C}^m$. Then for $i\in\mathbb{Z}[1,m]$, we have
\begin{align*}
\parallel\;N(r,0;\partial_{z_i}(f))\leq N(r,0;f)+\ol N(r,f)+o(T(r,f)).
\end{align*}
\end{lem}
\begin{proof} It is easy to verify that
\begin{align*}
N(r,\partial_{z_i}(f))\leq N(r,f)+\ol N(r,f),
\end{align*}
where $i\in\mathbb{Z}[1,m]$. Now using the first main theorem and Lemma \ref{L.1}, we get
\begin{align*} \parallel\;m(r,0,f)\leq m(r,0,\partial_{z_i}(f))+m\left(r,\frac{\partial_{z_i}(f)}{f}\right)=m(r,0;\partial_{z_i}(f))+o(T(r,f)),
\end{align*}
i.e.,
\begin{align*}
 \parallel\; N(r,0;\partial_{z_i}(f))\leq& T(r,\partial_{z_i}(f))-T(r,f)+N(r,0,f)+o(T(r,f))\\\leq&
N(r,\partial_{z_i}(f))m\left(r,\frac{\partial_{z_i}(f)}{f}\right)+m(r,f)-T(r,f)+N(r,0;f)+o(T(r,f))\\\leq& \ol N(r,f)+N(r,0;f)+o(T(r,f)).
\end{align*}
\end{proof}

\begin{lem}\label{l2.5} Let $g$ be a meromorphic function in $\mathbb{C}^m$. If $\partial^2_{z_i}(g(z))\equiv 0$ for all $i=1,2,\ldots,m$, then $g$ must be a polynomial in $\mathbb{C}^m$.
\end{lem}
\begin{proof} By the given condition, we have $\partial^2_{z_i}(g(z))\equiv 0$ for all $i=1,2,\ldots,m$ and so
\begin{align}\label{n.1} 
\frac{\partial^2g(z)}{\partial z_i^2}\equiv 0
\end{align}
for $i=1,2,\ldots,m$. We now want to prove that $g$ is a polynomial in $\mathbb{C}^m$. Here prove that $g$ is a polynomial in $\mathbb{C}^m$ by induction on the number $m$ of variables. If $m=1$, from (\ref{n.1}), we get that $g$ is a polynomial in $\mathbb{C}$. Let us suppose that $m=2$. Since $\frac{\partial^2g(z_1,z_2)}{\partial z_2^2}\equiv 0$, on integration, we have 
\begin{align}\label{n.2} 
g(z_1,z_2)=\phi_1(z_1)z_2+\phi_2(z_1),
\end{align}
where $\phi_i(z_1)$'s are entire functions in $\mathbb{C}$ in the variable $z_1$. Note that $\frac{\partial^2g(z_1,z_2)}{\partial z_1^2}\equiv 0$ and so (\ref{n.2}) gives 
\begin{align*}
\phi_1^{(2)}(z_1)z_2+\phi_2^{(2)}(z_1)\equiv 0,
\end{align*}
which shows that $\phi_1^{(2)}(z_1)\equiv 0$ and $\phi_2^{(2)}(z_1)\equiv 0$. Therefore on integration, we get
$\phi_1(z_1)=c_1z_1+c_2$ and $\phi_2(z_1)=d_1z_1+d_2$, where $c_1,c_2,d_1$ and $d_2$ are constants in $\mathbb{C}$. Therefore from (\ref{n.2}), we get 
\begin{align*}
g(z_1,z_2)=(c_1z_1+c_2)z_2+(d_1z_1+d_2),
\end{align*}
 which shows that $g(z_1,z_2)$ is a polynomial in $\mathbb{C}^2$. Now we fix $m\geq 2$ and assume that $g(z)$ is a polynomial for variables of number at most $m-1$. Since $\frac{\partial^2g(z)}{\partial z_m^2}\equiv 0$, we have 
\begin{align}\label{n.3} 
g(z_1,z_2,\ldots,z_m)= A(z_1,z_2,\ldots,z_{m-1})z_m+B(z_1,z_2,\ldots,z_{m-1}).
\end{align}

Now using (\ref{n.1}) to (\ref{n.3}), we get
\begin{align}\label{n.4} 
\frac{\partial^2 A(z_1,z_2,\ldots,z_{m-1})}{\partial z_i^2}z_m+\frac{\partial^2 B(z_1,z_2,\ldots,z_{m-1})}{\partial z_i^2}\equiv 0
\end{align}
for $i=1,2,\ldots,m-1$. Therefore (\ref{n.4}) yields
\begin{align}\label{n.5} 
\frac{\partial^2 A(z_1,z_2,\ldots,z_{m-1})}{\partial z_i^2}\equiv 0\;\;\text{and}\;\;\frac{\partial^2 B(z_1,z_2,\ldots,z_{m-1})}{\partial z_i^2}\equiv 0
\end{align}
for $i=1,2,\ldots,m-1$. Then by the induction assumptions we find from (\ref{n.5}) that both $A(z_1,z_2,\ldots,z_{m-1})$ and $B(z_1,z_2,\ldots,z_{m-1})$ are polynomials in the variables $z_1,z_2,\ldots,z_{m-1}$. Therefore from (\ref{n.3}), we get that $g(z_1,z_2,\ldots,z_m)$ is a polynomial in $z_1,z_2,\ldots,z_m$.
\end{proof}
\vspace{1.2mm}

For $r>0$ and $c=(c_1,c_2,\ldots,c_m)\in\mathbb{C}^m$, we define 
$\mathbb{D}(c;r)=\{z\in\mathbb{C}^m: |z_i-c_i|<r; i=1,2,\ldots,m\}$ and $\ol{\mathbb{D}(c;r)}=\{z\in\mathbb{C}^m: |z_i-c_i|\leq r; i=1,2,\ldots,m\}$. The Shilov boundary is given by $\mathbb{D}\langle c;r\rangle=\{z\in\mathbb{C}^m: |z_i-c_i|=r, i=1,2,\ldots,m\}$.
 We denote by $C_k(c_k,r)$ the boundary of $|z_k-c_k|<r$. Of course $C_k(c_k,r)$ is represented by the usual parametrization $\theta_k\to \gamma(\theta_k)=c_k+re^{i\theta_k}$, where $0\leq \theta_k\leq 2\pi$. Clearly $\mathbb{D}\langle c;r\rangle=C_1(c_1,r)\times\ldots\times C_m(c_m,r)$. Also for $r>0$, we define $\mathbb{C}^m[r]=\{z\in\mathbb{C}^m: ||z||\leq r\}$.

\subsection{\bf{Maximum principle}} Let $f(z)$ be a holomorphic function in a domain $D$ in $\mathbb{C}^m$. If $|f(z)|$ attains its maximum at a point of $D$, then $f(z)$ is constant in $D$.

Contrary to the case of one complex variable, in some domains $D$ in $\mathbb{C}^m\;(m>1)$ there exists a proper closed subset ${\bf e}$ of $\partial D$, where $\partial D$ denotes the boundary of the domain $D$ such that any holomorphic function $f(z)$ in $D$ with continuous boundary values attains its maximum modulus at a point of ${\bf e}$. Given $D\subset \mathbb{C}^m$, the smallest set ${\bf e}\subset \partial D$ with this property is called the Shilov boundary of $D$. For example, the Shilov boundary of a polydisk $|z_j|< r_j\;(j=1,\ldots,m)$ is the distinguished boundary $|z_j|=r_j\;(j = 1,\ldots,m)$. On the other hand, the Shilov boundary of an open ball $B$ is the topological boundary, the sphere $\partial B$.

\subsection{\bf{The function $A(r,f)$}}
Let $f(z)=u(x_1,y_1,\ldots,x_m,y_m)+\iota v(x_1,y_1,\ldots,x_m,y_m)$ be holomorphic in $\ol{\mathbb{D}(0;R)}$, where $R>0$. Let $z=\left(re^{\iota \theta_1},re^{\iota \theta_2},\ldots,re^{\iota \theta_m}\right)$, where $0\leq r\leq R$. Then 
\begin{align*}
 f(z)=f\left(re^{\iota \theta_1},re^{\iota \theta_2},\ldots,re^{\iota \theta_m}\right)=u(r,\theta_1,\theta_2,\ldots,\theta_m)+\iota v(r,\theta_1,\theta_2,\ldots,\theta_m).
 \end{align*}
 
Let $A(r,f)$ denote the maximum value of $\Re\{f(z)\}$ on $\mathbb{D}\langle 0;r\rangle$, i.e.,
\begin{align*}
 A(r,f)=\max\limits_{z\in \mathbb{D}\langle 0;r\rangle} \Re\{f(z)\}=\max\{u(r,\theta_1,\theta_2,\ldots,\theta_m): 0\leq \theta_i\leq 2\pi, i=1,2,\ldots,m\}.
\end{align*}

Clearly $u(r,\theta_1,\theta_2,\ldots,\theta_m)\leq A(r,f)$ for $0\leq \theta_i\leq 2\pi$, where $i=1,2,\ldots,m$. If $f(z)$ is constant, then $A(r)$ is also a constant. Suppose that $f(z)$ is non-constant.
Let $\phi(z)=e^{f(z)}$. Then $\phi(z)$ is a holomorphic function on $\ol{\mathbb{D}(0;R)}$. Now
\begin{align*}
|\phi(z)|=\left|e^{u(r,\theta_1,\theta_2,\ldots,\theta_m)}\right|=e^{u(r,\theta_1,\theta_2,\ldots,\theta_m)}.
\end{align*}

Let $0\leq r_{1}<r_{2}<R$. Since $\phi(z)$ is holomorphic in $\ol{\mathbb{D}(0;r_1)}$, the maximum value of $|\phi(z)|$ for $\ol{\mathbb{D}(0;r_1)}$ is attained on $\mathbb{D}\langle 0;r_1\rangle$, by maximum modulus theorem. Let $z_{1}=\left(r_1e^{\iota \theta_1},\ldots, r_1e^{\iota \theta_m}\right)$ be such a point on $\mathbb{D}\langle 0;r_1\rangle$, at which $|\phi(z_{1})|=\max\limits_{z\in \ol{\mathbb{D}(0;r_1)}}|\phi(z)|$.

Again since $\phi(z)$ is analytic in $\ol{\mathbb{D}(0;r_2)}$, the maximum value of $|\phi(z)|$ for $\ol{\mathbb{D}(0;r_2)}$ is attained on $\mathbb{D}\langle 0;r_2\rangle$. Let $z_{2}=\left(r_2e^{\iota \psi_1},\ldots, r_2e^{\iota \psi_m}\right)$ be such a point on $\mathbb{D}\langle 0;r_2\rangle$, at which $|\phi(z_{2})|=\max\limits_{z\in \ol{\mathbb{D}(0;r_2)}}|\phi(z)|$.
Since $r_{1}<r_{2}$, we have $|\phi(z_{1})|<|\phi(z_{2})|$, i.e., $\max\limits_{z\in \mathbb{D}\langle 0;r_1\rangle}|\phi(z)|<\max\limits_{z\in \mathbb{D}\langle 0;r_2\rangle}|\phi(z)|$ and so
\begin{align*} \exp \Big(\max\{u(r_1,\theta_1,\ldots,\theta_m): 0\leq \theta_i\leq 2\pi\}\Big)<\exp \Big(\max\{u(r_2,\theta_1,\ldots,\theta_m): 0\leq \theta_i\leq 2\pi\}\Big),
\end{align*}
i.e., $A(r_{1},f)<A(r_{2},f)$. This shows that $A(r,f)$ is steadily increasing function of $r$.

\subsection{\bf{The function $M(r,f)$}}
Let $f$ be a holomorphic function in $\ol{\mathbb{D}(0;R)}$, where $R>0$. For $0\leq \sqrt{m}r\leq R$, we define 
\begin{align*}
M(r,f)=\max\limits_{||z||=r}|f(z)|.
\end{align*}

The function $M(r,f)$ is called the growth function of $f$. Obviously $M(r,f)$ is steadily increasing function of $r$ and for a non-constant holomorphic function $f$ in $\mathbb{C}^m$, we have $M(r,f)\to \infty$ as $r\to \infty$.

\subsection{\bf{Schwarz's Lemma}} \cite[pp. 8]{gr} Let $f$ be holomorphic in $\ol{\mathbb{D}(0;r)}$ and suppose that $f$ is of total order $k$ at $0$ and that $|f(z)|\leq M$ for all $z\in \ol{\mathbb{D}(0;r)}$. Then 
\begin{align*}
|f(z)|\leq M\frac{||z||_{\infty}^k}{r^k},
\end{align*}
for all $z\in \ol{\mathbb{D}(0;r)}$, where $||z||_{\infty}=\max\{|z_k|: k=1,2,\ldots,m\}$.

\subsection{\bf{Borel-Caratheodery Lemma in several complex variables}}
\begin{lem}\label{t2.1} Suppose that $f$ is a holomorphic function in $\ol{\mathbb{D}(0;R)}\;(0<R<+\infty)$. Then
\begin{align*}
 M(r,f)\leq \frac{2r}{R-r}A(R,f)+\frac{R+r}{R-r} |f(0)|
\end{align*}
holds for $0\leq \sqrt{m}r< R$.
\end{lem}
\begin{proof} We consider the following three cases.
\vspace{1.2mm}

{\bf Case 1.} Assume that $f$ is a constant. Let $f(z)=\alpha+i\beta$, where $\alpha$ and $\beta$ are real constants. Clearly $|f(0)|=\sqrt{\alpha^{2}+\beta^{2}}$, $M(r,f)=\sqrt{\alpha^{2}+\beta^{2}}$ and $A(r,f)=\alpha$. Then
\begin{align*}
 \frac{2r}{R-r}A(R,f)+\frac{R+r}{R-r}|f(0)|-M(r,f)=\frac{2r}{R-r}\Big(\alpha+\sqrt{\alpha^{2}+\beta^{2}}\Big).
 \end{align*}

Since $\alpha+\sqrt{\alpha^{2}+\beta^{2}}\geq \alpha+|\alpha|\geq 0$, we get
\begin{align*}
M(r,f)\leq \frac{2r}{R-r}A(R,f)+\frac{R+r}{R-r}|f(0)|.
\end{align*}
\vspace{1.2mm}

{\bf Case 2.} Assume that $f$ is non-constant and $f(0)=0$. Clearly $A(0,f)=0=M(0,f)$. Since both $A(r,f)$ and $M(r,f)$ are steadily increasing functions of $r$ and so for $r>0$, we have $A(r,f)>0$ and $M(r,f)>0$. Let $f(z)=u(x_1,y_1,\ldots,x_m,y_m)+\iota v(x_1,y_1,\ldots,x_m,y_m)$. Clearly
\begin{align*}
 2A(R,f)-f(z)=(2A(R,f)-u)+\iota (-v)\;\;\text{and}\;\;\Re\{2A(R,f)-f(z)\}=2A(R,f)-u.
 \end{align*}

For $0<\sqrt{m}r<R$, we have $0<A(\sqrt{m}r,f)\leq A(R,f)$. Since $u\leq A(\sqrt{m}r,f)$, we have $u\leq A(R,f)$ and $u<2A(R,f)$. Consequently $A(R,f)-u\geq 0$ and $2A(R,f)-u>0$. Note that
\begin{align}\label{pd1} 
|2A(R,f)-f(z)|^{2}=(2A(R,f)-u)^{2}+v^{2}=4A(R,f)[A(R,f)-u]+u^{2}+v^{2}\geq u^{2}+v^{2}.
\end{align}

Let 
\begin{align}\label{pd2} 
\phi(z)=\frac{f(z)}{2A(R,f)-f(z)}.
\end{align}

Clearly $\phi(z)$ is holomorphic in $\ol{\mathbb{D}(0;R)}$ and $\phi(0)=0$. Therefore using (\ref{pd1}) to (\ref{pd2}), we get $|\phi(z)|\leq 1$, for all $z\in \ol{\mathbb{D}(0;R)}$. Then by Schwarz's Lemma, we have $|\phi(z)|\leq \frac{1.r}{R}$, i.e., 
\begin{align}\label{pd3}
|\phi(z)|\leq \frac{r}{R}
\end{align}
holds for all $z\in\mathbb{C}^m[r]$, where $\sqrt{m}r<R$. Now from (\ref{pd2}), we have 
\begin{align}\label{pd4} 
|f(z)|=\Big|\frac{2A(R,f)\phi(z)}{1+\phi(z)}\Big|\leq \frac{2A(R,f)|\phi(z)|}{1-|\phi(z)|}.
\end{align}

Therefore using (\ref{pd3}) to (\ref{pd4}), we have
\begin{align}\label{pd5} 
|f(z)|\leq \frac{2A(R,f)\frac{r}{R}}{1-\frac{r}{R}}=\frac{2r}{R-r}A(R,f)
\end{align}
for all $z\in\mathbb{C}^m[r]$, where $\sqrt{m}r<R$. Since $f(0)=0$, using maximum modulus theorem to (\ref{pd5}), we have 
\begin{align*}
 M(r,f)\leq \frac{2r}{R-r}A(R,f)+\frac{R+r}{R-r}|f(0)|
 \end{align*}
holds for $0\leq \sqrt{m}r<R$.\vspace{1.2mm}

{\bf Case 3.} Assume that $f$ is non-constant and $f(0)\not=0$. Let $\phi(z)=f(z)-f(0)$. Clearly $\phi(0)=0$. Using maximum modulus theorem to (\ref{pd5}) (replacing $f$ by $\phi$), we get
\begin{align}\label{pd6} 
\max\limits_{||z||=r}|\phi(z)|\leq \frac{2r}{R-r}\max\limits_{z\in \mathbb{D}\langle 0;R\rangle}\Re\{\phi(z)\},
\end{align}
where $\sqrt{m}r<R$. Now we see that
\begin{align*}
 \max\limits_{||z||=r}|\phi(z)|=\max\limits_{||z||=r}|f(z)-f(0)|\geq \max\limits_{||z||=r}|f(z)|-|f(0)|=M(r,f)-|f(0)|
 \end{align*}
and 
\begin{align*}
 \max\limits_{z\in \mathbb{D}\langle 0;R\rangle}\Re\{\phi(z)\}=\max\limits_{z\in \mathbb{D}\langle 0;R\rangle}\Re\{f(z)-f(0)\}\leq \max\limits_{z\in \mathbb{D}\langle 0;R\rangle}\Re\{f(z)\}+|f(0)|=A(R,f)+|f(0)|.
 \end{align*}

Then from (\ref{pd6}), we deduce that 
\begin{align*}
 M(r,f)\leq \frac{2r}{R-r}A(R,f)+\frac{R+r}{R-r}|f(0)|
 \end{align*}
holds for $0\leq \sqrt{m}r<R$.
\end{proof}

In 1995, Hu and Yang \cite{hy1} obtained the following result.
\begin{lem}\label{2A}\cite[Proposition 3.2]{hy1} Let $P$ be a non-constant entire function in $\mathbb{C}^m$. Then 
\begin{align*}
\rho(e^P)=
\begin{cases}
\deg(P), & \text{if $P$ is a polynomial,}\\
+\infty, & \text{otherwise.}
\end{cases}
\end{align*}
\end{lem}

\begin{lem}\cite[Lemma 2.5.24]{Noguchi-Winkelmann}\label{l2.1} Let $f:\mathbb{C}^m\to \mathbb{C}$ be an entire function. Then for $0<r<R$,
\begin{align*}
T(r,f)\leq \log^+ M(r,f)\leq \frac{1-\left(\frac{r}{R}\right)^2}{\left(1-\frac{r}{R}\right)^{2m}}T(R,f).
\end{align*}
\end{lem}

From Lemma \ref{l2.1}, one can easily prove that
\begin{align*}
\rho(f):=\limsup _{r \rightarrow \infty} \frac{\log^+ T(r, f)}{\log r}=\limsup _{r \rightarrow \infty} \frac{\log^+ \log^+ M(r, f)}{\log r}.
\end{align*}

Let $f=e^{h}$, where $h$ is an entire function in $\mathbb{C}^m$. With respect to the hyper-order of $f$, we establish the following result.

\begin{lem}\label{t2.2} Suppose $h$ is a non-constant entire function in $\mathbb{C}^m$ and $f=e^{h}$.  Then $\rho(h)=\rho_1(f)$.
\end{lem}
\begin{proof}
We define
\begin{align*}
M(r, h)=\max_{||z||=r} |h(z)|\quad \text{and}\quad
A(r, h)=\max\limits_{z\in \mathbb{D}\langle 0;r\rangle} \Re\{h(z)\}.
\end{align*}

Since $\Re\{h(z)\}\leq |h(z)|$, by the maximum modulus theorem, we get
\begin{align}\label{pd7} 
A(r, h)\leq \max\limits_{z\in \mathbb{D}\langle 0;r\rangle}|h(z)|\leq \max\limits_{||z||=\sqrt{m}r}|h(z)|=M(\sqrt{m}r,h).
\end{align}

Again by the maximum modulus theorem, we get
\begin{align}\label{pd7a}
\max_{z\in \mathbb{D}\langle 0;r\rangle} |f(z)|=\max_{z\in \mathbb{D}\langle 0;r\rangle} \left|e^{h(z)}\right|=\max_{z\in \mathbb{D}\langle 0;r\rangle} |e^{h(z)}|=e^{A(r,h)}.
 \end{align}
 
Since $M(r,f)=\max\limits_{||z||=r} |f(z)|\leq \max\limits_{z\in \mathbb{D}\langle 0;r\rangle}|f(z)|$, it follows from (\ref{pd7a}) that
$\log^+ M(r, f)\leq A(r, h)$. Now from Lemma \ref{l2.1} and (\ref{pd7}), we get
\begin{align*}
 T(r, f) \leq \log^+ M(r, f)\leq A(r, h) \leq M(\sqrt{m}r, h),
\end{align*}
from which we conclude that $\rho_1(f)\leq \rho(h)$.
Again by Lemma \ref{l2.1} (by taking $R=2\sqrt{m}r$), we have
\begin{align}\label{pd9} 
T(r, h) \leq \log^+ M(r, h)\leq \frac{1-\left(\frac{1}{2\sqrt{m}}\right)^2}{\left(1-\frac{1}{2\sqrt{m}}\right)^{2m}} T(2\sqrt{m}r, h).
\end{align}

Since $\max\limits_{z\in \mathbb{D}\langle 0;r\rangle}|f(z)|\leq \max\limits_{||z||=\sqrt{m}r} |f(z)|=M(\sqrt{m}r,f)$, it follows from (\ref{pd7a}) that $A(r, h)\leq \log^+ M(\sqrt{m}r, f)$.
Now using (\ref{pd9}) to Lemma \ref{t2.1}, we get
\begin{align*} M(r, h)<&\frac{2}{2\sqrt{m}-1}A(2\sqrt{m}r, h) +\frac{2\sqrt{m}+1}{2\sqrt{m}-1} |h(0)|\\\leq &\frac{2}{2\sqrt{m}-1}\log^+ M(2mr, f) + \frac{2\sqrt{m}+1}{2\sqrt{m}-1}|h(0)| \\
<& \frac{2}{2\sqrt{m}-1} \frac{1-\left(\frac{1}{2\sqrt{m}}\right)^2}{\left(1-\frac{1}{2\sqrt{m}}\right)^{2m}} T(4m\sqrt{m}r, f) +\frac{2\sqrt{m}+1}{2\sqrt{m}-1}|h(0)|,
\end{align*}
from which we conclude that $\rho(h)\leq \rho_1(f)$.
Finally, we have $\rho(h)=\rho_1(f)$.
\end{proof}

\begin{lem}\label{2L.1} Let $f$ be a non-constant entire function in $\mathbb{C}^m$ such that $\partial_{z_i}(f)\not\equiv 0$ for all $i\in\mathbb{Z}[1,m]$. Then 
\begin{align*}
\max\{\rho(\partial_{z_1}(f)),\ldots,\rho(\partial_{z_m}(f))\}=\rho(f).
\end{align*}
\end{lem}
\begin{proof} First we suppose that $f(z)$ is a polynomial. Then $\partial_{z_1}(f(z)),\partial_{z_2}(f(z)),\ldots, \partial_{z_m}(f(z))$ are also polynomials. Since $T(r,f)=O(\log r)$ and $T(r,\partial_{z_i}(f))=O(\log r)$ for all $i\in\mathbb{Z}[1,m]$, it follows that $\rho(f)=0$ and $\rho(\partial_{z_i}(f))=0$ for all $i\in\mathbb{Z}[1,m]$. Therefore 
\begin{align*}
\max\{\rho(\partial_{z_1}(f)),\ldots,\rho(\partial_{z_m}(f))\}=\rho(f).
\end{align*}

Next we suppose that $f(z)$ is a transcendental entire function. Then by Proposition 3.3 \cite{hy1}, we have $\rho(\partial_{z_i}(f))\leq \rho(f)$ and so
\begin{align}\label{sm2} 
\max\{\rho(\partial_{z_1}(f)),\ldots,\rho(\partial_{z_m}(f))\}\leq \rho(f).
\end{align}

Let $\tilde z, c\in \ol{\mathbb{D}(0;\sqrt{m} r)}$, where $\tilde z=(\tilde z_1,\ldots,\tilde z_m)$ and $c=(c_1,\ldots,c_m)$. For fixed $\tilde z$, $c$, let $w(t)=c+t(\tilde z-c)$ for all $t\in [0,1]$. Let $F: [0,1]\to \mathbb{C}$ be defined by $F(t)=f(w(t))=f(c+t(\tilde z-c))$. Clearly $F(1)=f(\tilde z)$, $F(0)=f(c)$ and 
\begin{align}\label{sm4} 
F^{(1)}(t)=\frac{\partial F(t)}{\partial t}=\sideset{}{_{i=1}^m}{\sum} \frac{\partial F(t)}{\partial z_i} (\tilde{z}_i-c_i)=&\sideset{}{_{i=1}^m}{\sum} \frac{\partial f(c+t(\tilde z-c))}{\partial z_i}.(\tilde{z}_i-c_i)\\=&\sideset{}{_{i=1}^m}{\sum} \partial_{z_i}(f(c+t(\tilde z-c))).(\tilde{z}_i-c_i)\nonumber.
\end{align}

We know that $F(1)-F(0)=\int\limits_{0}^1 F^{(1)}(t)dt$ and so from (\ref{sm4}), we have
\begin{align}\label{sm4a} 
|f(\tilde z)-f(c)|\leq& \sideset{}{_{i=1}^m}{\sum}\int\limits_0^1\left|\partial_{z_i}(f(c+t(\tilde z-c))).(\tilde{z}_i-c_i)\right|dt\\\leq& 
\sqrt{m}r \sideset{}{_{i=1}^m}{\sum}\max\limits_{0\leq t\leq 1}\left|\partial_{z_i}(f(c+t(\tilde z-c)))\right|\nonumber \\\leq&
\sqrt{m}r \sideset{}{_{i=1}^m}{\sum}\max\limits_{\mathbb{C}^m[r]}\left|\partial_{z_i}(f(z)))\right|.\nonumber
\end{align}

Clearly (\ref{sm4a}) holds for all $\tilde z\in \ol{\mathbb{D}(0;\sqrt{m} r)}$ and so by the maximum modulus theorem, we get
\begin{align*}
 \max\limits_{\mathbb{C}^m[r]} |f(z)|\leq \sqrt{m}r \sideset{}{_{i=1}^m}{\sum}\max\limits_{\mathbb{C}^m[r]}\left|\partial_{z_i}(f(z)))\right|+|f(c)|,
 \end{align*}
i.e.,
\begin{align}\label{sm5} 
M(r,f)\leq \sqrt{m}r \sideset{}{_{i=1}^m}{\sum}M(r,\partial_{z_i}(f))+|f(c)|.
\end{align}

By the definition of order, for a given $\varepsilon>0$, there exists $R(\varepsilon)>0$ such that
\begin{align*}
M(r,\partial_{z_i}(f))<e^{r^{\rho(\partial_{z_i}(f))+\varepsilon}}\;\;\forall\;r>R(\varepsilon)
\end{align*}
for all $i\in\mathbb{Z}[1,m]$ and so from (\ref{sm5}), we get
\begin{align*}
 M(r,f)\leq 2\sqrt{m}m r e^{r^{d+\varepsilon}},
 \end{align*}
where $d=\max\{\rho(\partial_{z_1}(f)),\ldots,\rho(\partial_{z_m}(f))\}$. Consequently, $\rho(f)\leq d+\varepsilon$. Since $\varepsilon>0$ was arbitrary, it follows that $\rho(f)\leq \max\{\rho(\partial_{z_1}(f)),\ldots,\rho(\partial_{z_m}(f))\}$ and so from (\ref{sm2}), we have
\begin{align*}
\max\{\rho(\partial_{z_1}(f)),\ldots,\rho(\partial_{z_m}(f))\}=\rho(f).
\end{align*}
\end{proof}

\section {{\bf Proof of Theorem \ref{t3.1}}}
By the given condition, we have
\begin{align}\label{snn1} 
\partial_{z_i}(f(z))=e^{\alpha(z)}f(z)
\end{align}
for all $i\in\mathbb{Z}[1,m]$. Clearly $f(z)$ and $\partial_{z_i}(f(z))$ share $0$ CM for all $i\in\mathbb{Z}[1,m]$. Let $z_0$ be a zero of $f$ of multiplicity $p$. Then by the definition of zero, we can say that $z_0$ must be a zero of $\partial_{z_i}(f)$ of multiplicity atmost $p-1$ for atleast one $i\in\mathbb{Z}[1,m]$. Note that (\ref{snn1}) holds for all $i\in\mathbb{Z}[1,m]$. Therefore from (\ref{snn1}), we get a contradiction. Hence $f$ has no zeros and so from (\ref{snn1}), we see that $\partial_{z_i}(f(z))$ has no zeros for all $i\in\mathbb{Z}[1,m]$. Let $f(z)=e^{\beta(z)}$, where $\beta(z)$ is a non-constant entire function in $\mathbb{C}^m$. Now using Lemma \ref{t2.2}, we conclude that $\rho(\beta)=\rho_1(f)$ and so $\rho(\beta)\not\in\mathbb{N}\cup\{\infty\}$. Clearly $\partial_{z_i}(f(z))=\partial_{z_i}(\beta(z))e^{\beta(z)}$ and so 
\begin{align}\label{sn0}
\partial_{z_i}(f(z))=\partial_{z_i}(\beta(z))e^{\beta(z)}=\partial_{z_i}(\beta(z))f(z).
\end{align}

Since $\partial_{z_i}(f(z))$ has no zeros, $\partial_{z_i}(\beta(z))$ has no zeros for all $i\in\mathbb{Z}[1,m]$. 
Then there exist entire functions $\delta_1(z),\ldots, \delta_m(z)$ in $\mathbb{C}^m$ such that 
\begin{align}\label{sn1}
\partial_{z_i}(\beta(z))=e^{\delta_i(z)}
\end{align}
for $i\in\mathbb{Z}[1,m]$. Since $\rho(\partial_{z_i}(\beta))\leq \rho(\beta)<+\infty$, using Lemma \ref{2A}, we get from (\ref{sn1}) that $\delta_1(z),\ldots,\delta_m(z)$ are polynomials in $\mathbb{C}^m$ such that $\rho(\partial_{z_i}(\beta))=\deg(\delta_i)$
for all $i\in\mathbb{Z}[1,m]$. Again since $\rho(\partial_{z_i}(\beta))\leq \rho(\beta)$, using Lemma \ref{2L.1}, we have
\begin{align*}
\max\{\rho(\partial_{z_1}(\beta)),\ldots,\rho(\partial_{z_m}(\beta))\}= \rho(\beta)\not\in\mathbb{N}\cup\{\infty\}
\end{align*}
and so
\begin{align*}
\max\{\deg(\delta_1),\ldots,\deg(\delta_m)\}=\rho(\beta)\not\in\mathbb{N}\cup\{\infty\},
\end{align*}
from which we conclude that $\delta_1(z),\delta_2(z)\ldots,\delta_m(z)$ are constants. Consequently from (\ref{sn1}), we see that 
$\partial_{z_1}(\beta(z)), \partial_{z_2}(\beta(z)),\ldots, \partial_{z_m}(\beta(z))$ are also constants. Let
\begin{align}\label{sn.6} 
\partial_{z_i}(\beta(z))=A_i
\end{align}
for all $i\in\mathbb{Z}[1,m]$. Now from (\ref{snn1}), (\ref{sn0}) and (\ref{sn.6}), we see that $\alpha(z)$ reduces to a constant, say $c$ and $e^c=A_1=A_2=\ldots=A_m=A$. Clearly $\beta(z)$ has the Taylor expansion near $(0,0,\ldots,0)$,
\begin{align}\label{sn.7} 
\beta(z)=\sum\limits_{i_1,\ldots,i_m=0}^{\infty} a_{i_1\ldots i_m} z_1^{i_1}\ldots z_m^{i_m},
\end{align}
where the coefficient $a_{i_1\ldots i_m}$ is given by
\begin{align}\label{sn.8} 
a_{i_1\ldots i_m}=\frac{1}{i_1!\ldots i_m!}\frac{\partial^{|I|}\beta(0,0,\ldots,0)}{\partial z_1^{i_1}\cdots \partial z_m^{i_m}}.
\end{align}

Now using (\ref{sn.6}) and (\ref{sn.8}) to (\ref{sn.7}), we get $\beta(z)=B_0+A(z_1+ \cdots+z_m)$,
where $B_0=a_{0\ldots 0}=\beta(0,0,\ldots,0)$. Finally $f(z_1,\ldots ,z_m)=c_1\exp(A(z_1+ \cdots+z_m))$,
where $c_1=\exp(B_0)$.

\section {{\bf Proof of Theorem \ref{t3.2}}}
By the given conditions, we have $\parallel N(r,0;\partial_{z_i}(f))=o(T(r,f))$ and 
\bea\label{0x} \partial_{z_i}(f(z))-a=e^{\alpha(z)}(f(z)-a)\eea
for all $i\in\mathbb{Z}[1,m]$. Clearly $f$ and $\partial_{z_i}(f)$ share $a$ CM for all $i\in\mathbb{Z}[1,m]$. 
Now we consider the following two cases.\vspace{1.2mm}

{\bf Case 1.} Let $\alpha(z)$ be a constant. Suppose $e^{\alpha(z)}=A$. Then from (\ref{0x}), we have
\begin{align}\label{x1}
\partial_{z_i}(f(z))-a=A(f(z)-a)
\end{align}
for all $i\in\mathbb{Z}[1,m]$. Let $g(z)=f(z)-a$. Then $\partial_{z_i}(g)=\partial_{z_i}(f)$ and so from (\ref{x1}), we get
\begin{align}\label{x2} 
\partial_{z_i}(g(z))=A\left(g(z)+\frac{a}{A}\right)
\end{align}
for all $i\in\mathbb{Z}[1,m]$. Let $z_0$ be a zero of $g+\frac{a}{A}$ of multiplicity $p$. Then by the definition of zero, we can say that $z_0$ must be a zero of $\partial_{z_i}(g)$ of multiplicity atmost $p-1$ for atleast one $i\in\mathbb{Z}[1,m]$. Note that (\ref{x2}) holds for all $i\in\mathbb{Z}[1,m]$. Therefore from (\ref{x2}), we get a contradiction. Hence $g+\frac{a}{A}$ has no zeros.
Let us take 
\begin{align*}
g(z)+\frac{a}{A}=e^{\beta(z)},
\end{align*}
where $\beta(z)$ is an entire function in $\mathbb{C}^m$. Then from (\ref{x2}), we have 
$\partial_{z_i}(\beta(z))=A$,
for all $i\in\mathbb{Z}[1,m]$. Now proceeding in the same way as done in the proof of Theorem \ref{t3.1}, one can easily deduce that 
\begin{align*}
g(z)+\frac{a}{A}=c_1\exp(A(z_1+ \cdots+z_m)),
\end{align*}
where $c_1$ is a non-zero constant. Consequently
\begin{align*}
f(z)=c_1\exp(A(z_1+ \cdots+z_m))+a-\frac{a}{A}.
\end{align*}
\vspace{1.2mm}

{\bf Case 2.} Let $\alpha(z)$ be non-constant. Note that if $\partial_{z_i}^2(f)\equiv 0$ for all $i\in\mathbb{Z}[1,m]$, then by Lemma \ref{l2.5}, we conclude that that $f(z)$ is a polynomial in $\mathbb{C}^m$. In this case, from (\ref{x1}), we get a contradiction. Hence $\partial_{z_i}^2(f)\not\equiv 0$ for atlaest one $i\in\mathbb{Z}[1,m]$. Suppose $\partial_{z_k}^2(f)\not\equiv 0$. Since $\parallel N(r,0;\partial_{z_k}(f))=o(T(r,f))$, by Lemma \ref{L.6}, we deduce that
\begin{align}
\label{x0} \parallel\;N(r,0;\partial_{z_k}^2(f))=o(T(r,f)).
\end{align}

Now in view of (\ref{x0}) and using Lemma \ref{L.1}, we get
\begin{align}
\label{x00} \parallel\;T\left(r, \frac{\partial_{z_j z_i}^2(f)}{\partial_{z_i}(f)}\right)=o(T(r,f)),
\end{align}
for all $i,j\in\mathbb{Z}[1,m]$.
Let
\begin{align}
\label{yy0}F=\frac{\partial_{z_k}^2(f)}{\partial_{z_k}(f)}\;\;\text{and}\;\;G=\left(\frac{\partial_{z_k}(f)-a}{f-a}\right)^2.
\end{align}

We consider the following two sub-cases.\vspace{1.2mm}

{\bf Sub-case 2.2.} Let $F$ and $G$ be linearly independent. By Corollary 1.40 \cite{HLY}, there is $l \in \mathbb{Z}[1, m]$ such that 
\begin{align*}
W(F,G)=\left|\begin{array}{ll} F&G\\\partial_{z_l}(F)&\partial_{z_l}(G)\end{array}\right|\not\equiv 0.
\end{align*}

If we take $H=-\frac{W}{F G}$, then from (\ref{yy0}), we get
\begin{align}\label{yy1} 
H=\frac{\partial^3_{z_l z_k^2}(f)}{\partial_{z_k}^2(f)}-\frac{\partial^2_{z_lz_k}(f)}{\partial_{z_k}(f)}-2\left(\frac{\partial^2_{z_lz_k}(f)}{\partial_{z_k}(f)-a}-\frac{\partial_{z_l}(f)}{f-a}\right)\not\equiv 0,
\end{align}
where $\partial^3_{z_lz_k^2}(f(z))=\frac{\partial^3 f(z)}{\partial z_l\partial z_k^2}$ and $l,k\in\mathbb{Z}[1,m]$.

Let $z^0$ be a zero of $f-a$. By the given condition, we have $\parallel N(r,0;\partial_{z_i}(f))=o(T(r,f))$ for all $i\in\mathbb{Z}[1,m]$ and by (\ref{x0}), we have $\parallel\;N(r,0;\partial_{z_k}^2(f))=o(T(r,f))$. We assume that $\partial_{z_i}(f(z_0))\neq 0$ for all $i\in\mathbb{Z}[1,m]$ and $\partial^2_{z_k}(f(z_0))\neq 0$, otherwise the counting function of those zeros of $f-a$ which are the zeros of $\partial_{z_i}(f)$ for all $i\in\mathbb{Z}[1,m]$ and $\partial^2_{z_k}(f)$ is equal to $o(T(r,f))$.

If $z^0=(z^0_1,z^0_2,\ldots,z^0_m)$, then in a neighborhood of $z^0$, we can expand $f(z)-a$ as a convergent series of homogeneous polynomials in $z-z^0$:
\begin{align}\label{yy2} 
f(z)-a=\sideset{}{_{n=1}^{\infty}}{\sum}P_n(z-z^0).
\end{align}

Here $P_n$ is a homogeneous polynomial of degree $n$ and $P_1\not\equiv 0$. Since $f(z)$ and $\partial_{z_i}(f(z))$ share $a$ CM, from (\ref{yy2}), we get 
\begin{align}\label{yy2a}
\partial_{z_i}(P_1(z-z^0))=a
\end{align}
for all $i\in\mathbb{Z}[1,m]$ and so
\begin{align}\label{yy3} 
\partial_{z_i}(f(z))-a=\partial_{z_i}(P_2(z-z^0))+\partial_{z_i}(P_3(z-z^0))+\partial_{z_i}(P_4(z-z^0))+\ldots,
\end{align}
\begin{align}\label{yy4} 
\partial^2_{z_l z_k}(f(z))=\partial^2_{z_lz_k}(P_2(z-z^0))+\partial^2_{z_lz_k}(P_3(z-z^0))+\partial^2_{z_lz_k}(P_4(z-z^0))+\ldots,
\end{align}
\begin{align}
\label{yy5} \partial^2_{z_k}(f(z))=\partial^2_{z_k^2}(P_2(z-z^0))+\partial^2_{z_k^2}(P_3(z-z^0))+\partial^2_{z_k^2}(P_4(z-z^0))+\ldots
\end{align}
and 
\begin{align}\label{yy6} 
\partial^3_{z_lz_k^2}(f(z))=\partial^3_{z_lz_k^2}(P_3(z-z^0))+\partial^3_{z_lz_k^2}(P_4(z-z^0))+\ldots,
\end{align}
where $\partial^2_{z_k}(P_2(z-z^0))\neq 0$ and $\partial^3_{z_lz_k^3}(P_3(z-z^0))$ are constants. 
Let us take
\begin{align*}
e^{\alpha(z)}=c_0+Q_1(z-z^0)+Q_2(z-z^0)+\ldots,
\end{align*}
where $c_0$ is a non-zero constant and $Q_n$ is a homogeneous polynomial of degree $n$. Clearly from (\ref{0x}), we have
\begin{align}\label{yy8} 
\partial_{z_i}(f(z))-a=(c_0+Q_1(z-z^0)+Q_2(z-z^0)+\ldots)(f(z)-a),
\end{align}
for all $i\in\mathbb{Z}[1,m]$.
Now using (\ref{yy2}) and (\ref{yy3}) to (\ref{yy8}), we get
\begin{align}\label{yy9}
\partial_{z_i}(P_2(z-z^0))=c_0P_1(z-z^0),
\end{align}
and
\begin{align}\label{yy10} 
\partial_{z_i}(P_3(z-z^0))=c_0P_2(z-z^0)+P_1(z-z^0)Q_1(z-z^0),
\end{align}
for all $i\in\mathbb{Z}[1,m]$.
By the homogeneity of $P_3(z-z^0)$, we have
\begin{align*}
\sideset{}{_{i=1}^m}{\sum} (z_i-z_i^0)\partial_{z_i}(P_3(z-z^0))=3P_3(z-z^0)
\end{align*}
and so from (\ref{yy10}), we get
\begin{align}\label{yy10a}
m\partial_{z_j}(P_3(z-z^0))\sideset{}{_{i=1}^m}{\sum} (z_i-z_i^0)=3P_3(z-z^0),
\end{align}
for all $j\in\mathbb{Z}[1,m]$. Now from (\ref{yy10a}), we get
\begin{align}\label{yy10b} 
P_3(z-z^0)=d\left(\sideset{}{_{i=1}^m}{\sum} (z_i-z_i^0)\right)^3.
\end{align}
where $d$ is a non-zero constant. Clearly from (\ref{yy10b}), we have
\begin{align}\label{yy10c} 
\partial_{z_l z_k^2}^3(P_3(z-z^0))=\partial_{z_l}^3(P_3(z-z^0)).
\end{align}

Therefore using (\ref{yy2})-(\ref{yy6}), (\ref{yy10c}) to (\ref{yy10}), we get
\begin{align}\label{yy11} 
\partial^3_{z_l z_k^2}(f(z))\partial_{z_k}(f(z))-\partial^2_{z_lz_k}(f(z))\partial_{z_k}^2(f(z))=
2a^2\partial_{z_l}(Q_1(z-z^0))+\ldots
\end{align}

Again using (\ref{yy2})-(\ref{yy4}), (\ref{yy9}) and (\ref{yy10}), we have
\begin{align}\label{yy12} &(f(z)-a)\partial_{z_lz_k}^2(f(z))-\partial_{z_l}(f(z))(\partial_{z_k}(f(z))-a)\\=&P_1(z-z^0)\partial_{z_lz_k}^2(P_3(z-z^0))+ac_0P_2(z-z^0)-a\partial_{z_k}(P_3(z-z^0))\nonumber\\&-c_0P_1(z-z^0)\partial_{z_l}(P_2(z-z^0))+\ldots\nonumber\\=&
P_1^2(z-z^0)\partial_{z_l}(Q_1(z-z^0))+\ldots\nonumber
\end{align}

Now using (\ref{yy3}), (\ref{yy9}), (\ref{yy11}) and (\ref{yy12}) to (\ref{yy1}), we get  
$H(z^0)=0$ and so $H(z)$ is holomorphic at $z^0$. So $\parallel N(r,H)=o(T(r,f))$. Using Lemma \ref{L.1} to (\ref{yy1}), we get $\parallel m(r,H)=o(T(r,f))$ and so $\parallel T(r,H)=o(T(r,f))$. Now using the first main theorem, we get
\begin{align}\label{yy20} 
N(r,a;f)\leq N(r,0;H)\leq T(r,H)=o(T(r,f)).
\end{align}

Since $f$ and $\partial_{z_k}(f)$ share $a$ CM, using Lemma \ref{L.2}, we get
\begin{align*} \parallel T(r,\partial_{z_k}(f))\leq \ol N(r,0;\partial_{z_k}(f))+\ol N(r,a;\partial_{z_k}(f))+o(T(r,\partial_{z_k}(f)))=o(T(r,f))
\end{align*}
and so in view of the first main theorem and using Lemma \ref{L.1}, we have
\begin{align}\label{w0} 
m(r,a;f)\leq m(r,0;\partial_{z_k}(f))\leq T(r,\partial_{z_k}(f))=o(T(r,f)).
\end{align}

Finally view of (\ref{yy20}) and (\ref{w0}) and using the first main theorem, we get $\parallel T(r,f)=o(T(r,f))$, which is impossible.\vspace{1.2mm}

{\bf Sub-case 2.2.} Let $F$ and $G$ be linearly dependent. Then there exists $C\in\mathbb{C}\backslash \{0\}$ such that
\begin{align}\label{w1} 
C\frac{\partial_{z_k}^2(f)}{\partial_{z_k}(f)}=\left(\frac{\partial_{z_k}(f)-a}{f-a}\right)^2.
\end{align}

Now from (\ref{0x}), we get
\begin{align}\label{w2} 
\frac{\partial_{z_k}^2(f)}{\partial_{z_k}(f)}=\frac{\partial_{z_k}(e^{\alpha})(f-a)}{\partial_{z_k}(f)}+\frac{\partial_{z_k}(f)-a}{f-a}.
\end{align}

Let $z^0$ is a zero of $f-a$ such that $\partial_{z_k}(f(z_0))\neq 0$ and $\partial^2_{z_k}(f(z_0))\neq 0$. Then from (\ref{w1}) and (\ref{w2}), we get
\begin{align*}
\frac{\partial_{z_k}^2(f(z^0))}{\partial_{z_k}(f(z^0))}=C
\end{align*}
and so in view of (\ref{x00}) and using the first main theorem, we get
\begin{align}\label{w3}
N(r,a;f)\leq N\left(r,C;\frac{\partial_{z_k}^2(f)}{\partial_{z_k}(f)}\right)\leq T\left(r,\frac{\partial_{z_k}^2(f)}{\partial_{z_k}(f)}\right)=o(T(r,f)).
\end{align}

Now using (\ref{w0}), (\ref{w3}) and the first main theorem, we get $\parallel T(r,f)=o(T(r,f))$, which is absurd.

\medskip
{\bf Statements and declarations:}

\smallskip
\noindent \textbf {Conflict of interest:} The authors declare that there are no conflicts of interest regarding the publication of this paper.

\smallskip
\noindent{\bf Funding:} There is no funding received from any organizations for this research work.

\smallskip
\noindent \textbf {Data availability statement:}  Data sharing is not applicable to this article as no database were generated or analyzed during the current study.

\end{document}